\documentclass[12pt,centertags,oneside]{amsart}
\usepackage{a4}
\usepackage{amsmath,amstext,amsthm,amssymb,amsfonts,amscd}
\usepackage{mathrsfs}
\usepackage[colorlinks=false]{hyperref}
\usepackage[capitalise]{cleveref}
\usepackage{latexsym,bm,graphicx}
\usepackage[T1]{fontenc}
\usepackage[latin1]{inputenc}
\usepackage{color,tocvsec2}
\usepackage{slashed}
\usepackage{typearea}
\usepackage{charter}
\usepackage{enumerate}
\usepackage{lscape}
\usepackage[all]{xy}
\usepackage{tikz-cd}
\usepackage{rotating}
\usetikzlibrary{matrix,arrows,decorations.pathmorphing}

\usepackage{multicol}        
\makeindex             

\usepackage[a4paper,width=16.2cm,top=2.5cm,bottom=2.5cm]{geometry}

\theoremstyle{plain}
\newtheorem{theorem}{Theorem}[section]
\newtheorem{proposition}[theorem]{Proposition}
\newtheorem{lemma}[theorem]{Lemma}
\newtheorem{corollary}[theorem]{Corollary}

\newtheorem{conjecture}[theorem]{Conjecture}

\theoremstyle{definition} 
\newtheorem{definition}[theorem]{Definition}
\newtheorem{example}[theorem]{Example}
\theoremstyle{remark} 
\newtheorem{remark}[theorem]{Remark}

\numberwithin{equation}{section}
\numberwithin{figure}{section}

\makeatletter
\@namedef{subjclassname@2020}{\textup{2020} Mathematics Subject Classification}
\makeatother

\begin{document}
\title[]{On the long neck principle and width estimates \\ for initial data sets}

\author{Daoqiang Liu}

\address{School of Mathematical Sciences,
     Capital Normal University, 100048,
     Beijing, China}
\email{\href{mailto:dqliumath@cnu.edu.cn}{dqliumath@cnu.edu.cn}}
\urladdr{\href{https://www.dqliu.cn}{www.dqliu.cn}}


\date{}


\begin{abstract}
In this paper, we prove the long neck principle, band width estimates, and width inequalities of the geodesic collar neighborhoods of the boundary in the setting of general initial data sets for the Einstein equations, subject to certain energy conditions corresponding to the lower bounds of scalar curvature on Riemannian manifolds. Our results are established via the spinorial Callias operator approach. 
\end{abstract}

\maketitle


\section{Introduction}

In recent articles, Gromov proposed the following conjectures concerning the geometry of scalar curvature with a lower bound:

\begin{conjecture}[{\cite[p. 87, \text{Long neck problem}]{Gro19}}]\label{conj:long-neck-problem}
Let $(M,g_M)$ be a compact $m$-dimensional Riemannian manifold with boundary with scalar curvature ${\rm scal}_M\geq m(m-1)$. Suppose that $\Phi: M\to S^m$ is a smooth area non-increasing map which is locally constant near the boundary. If the distance with respect to the metric $g_M$
\begin{equation}
{\rm dist}_{g_M}({\rm supp}({\rm d}\Phi), \partial M) \geq \frac{\pi}{m},
\end{equation}
then ${\rm deg}(\Phi)=0$.
\end{conjecture}

\begin{conjecture}[{\cite[11.12, Conjecture C]{Gro18}}, Band width estimates]\label{conj:band-width}
Let $N$ be a closed manifold of dimension $m-1\neq 4$ such that $N$ does not admit a metric of positive scalar curvature. Let $g_M$ be a Riemannian metric $M=N\times [-1,1]$ such that scalar curvature ${\rm scal}_M \geq m(m-1)$. Then
\begin{equation}
{\rm width}(M, g_M)\leq \frac{2\pi}{m},
\end{equation}
where ${\rm width}(M,g_M):={\rm dist}_{g_M}(N\times \{-1\}, N\times \{1\} )$ is the distance with respect to $g_M$.
\end{conjecture}

\begin{conjecture}[{\cite[11.12, Conjecture D']{Gro18}}, Geodesic collar neighborhood problem]\label{conj:geodesic-collar-neighborhood}
Let $W$ be a closed $m$-dimensional manifold such that $W$ minus a point does not admit a complete metric of positive scalar curvature. Let $M$ be the manifold with boundary obtained from $W$ by removing a small $m$-dimensional open ball. Let $g_M$ be a Riemannian metric on $M$ such that ${\rm scal}_{M}\geq {\rm scal}_0>0$. Then there exists a constant $c>0$ such that if there exists an open geodesic collar neighborhood $\mathcal{N}$ of width $\rho$, then
\begin{equation}
\rho\leq\frac{c}{\sqrt{{\rm scal}_0}}.
\end{equation} 
Here observe that $M$ is a manifold with boundary $\partial M=S^{m-1}$ which is homeomorphic to $S^{m-1}\times [0,1]$. For $\rho>0$ small enough, an open geodesic collar neighborhood $\mathcal{N}$ of width $\rho$ is the $\rho$-neighborhood of $\partial M$ with respect to $g_M$.
\end{conjecture}

The major progress on Conjectures \ref{conj:long-neck-problem}, \ref{conj:band-width} and \ref{conj:geodesic-collar-neighborhood} was made by Cecchini \cite{Cec20}, Zeidler \cite{Zei20,Zei22}, Cecchini-Zeidler \cite{CZ22} via spinorial methods involving Callias operators. On the other hand, Conjecture \ref{conj:band-width} was first proved by Gromov \cite{Gro18} for torical bands and by Zhu \cite{Zhu21} for overtorical bands and extended by R\"{a}de \cite{Rad22} for more general warped products using a version of the minimal surface approach called $\mu$-bubble techniques in dimension $n\leq 7$. Moreover, Guo-Xie-Yu \cite{GXY23} proved similar (non-sharp) band estimates by quantitative K-theory. In \cite{HKKZ22}, Hirsch, Kazaras, Khuri and Zhang gave band width estimates in dimension three by spacetime harmonic functions. Before that, Chai-Wan proved results of this type for initial data sets with constant mean curvature in \cite{CW22}. Recently, Hirsch-Kazaras-Khuri-Zhang \cite{HKKZ23} established various spectral band width inequalities. With the spirit to \cite{CZ22,HKKZ23}, the author also gave spectral results of Conjectures \ref{conj:long-neck-problem} and \ref{conj:geodesic-collar-neighborhood} again via spinors in \cite{Liu23}.

The main purpose of this paper is to prove the variations of the results about Conjectures \ref{conj:long-neck-problem}, \ref{conj:band-width} and \ref{conj:geodesic-collar-neighborhood} in the setting of initial data sets for the Einstein equations. In the first part of the paper, we will prove the following theorem that affirms the long neck principle in \cite{Liu23} for an initial data set under the relative dominant energy condition, relative to a background manifold. We refer to Sections \ref{sec:preliminaries} and \ref{sec:long-neck-principle} for precise statements of terms used below.
\begin{theorem}\label{thm:IDS-long-neck-principle}
    Let $(N, g_N)$ be a compact $n$-dimensional Riemannian manifold with non-negative curvature operator $\mathcal{R}_N\geq 0$ on $\Lambda^2TN$, where $n\geq 2$ even. Let $(M, g_M, k)$ be a compact $m$-dimensional spin initial data set (Definition~\ref{defn:IDS}) with non-empty boundary such that $\mu-|J|_{g_M}\geq \mathcal{Q} >0$ on $M\setminus {\rm supp}({\rm d}\Phi)$, where $m=n+4t$ for $t\in\mathbb{Z}_{\geq 0}$, $\mathcal{Q}$ is a constant and $\Phi: M\to N$ is a smooth spin map which is locally constant near the boundary $\partial M$. Assume that
\begin{itemize}
    \item[(i)] the Euler characteristic $\chi(N)$ of $N$ is non-vanishing;
    \item[(ii)] the relative dominant energy condition (Definition~\ref{defn:relative-DEC})
    \begin{equation}\label{eq:scalar-comparison}
      \mu-|J|_{g_M} \geq \frac{1}{2} ( {\rm scal}_N \circ \Phi ) \cdot \| \Lambda^2 {\rm d}\Phi \|
    \end{equation}
    on ${\rm supp}({\rm d}\Phi)$;  
    \item[(iii)] the length of the neck
    \begin{equation}\label{eq:spectral-bound-of-neck}
       {\rm dist}_{g_M}({\rm supp}({\rm d}\Phi),\partial M) > \pi \sqrt{\frac{m-1}{2 m \mathcal{Q} }}.
    \end{equation}
\end{itemize}
Then ${\rm deg}_{\widehat{A}}(\Phi)=0$.
\end{theorem}
In particular, we deduce the following direct consequences of Theorem \ref{thm:IDS-long-neck-principle} by applying \cite[Theorem~0.1]{GS02}.

\begin{corollary}\label{cor:IDS-long-neck-relative-to-symmetric-spaces}
Let $(N,g_N)$ be a compact $n$-dimensional Riemannian manifold with non-negative curvature operator $\mathcal{R}_N\geq 0$ on $\Lambda^2TN$, where $n\geq 2$ even, such that the universal covering of $N$ is homeomorphic to a symmetric space $G/H$ of compact type with ${\rm rank}\ G={\rm rank}\ H$. Let $(M,g_M)$ be a compact $m$-dimensional spin initial data set with non-empty boundary such that $\mu-|J|_{g_M}\geq \mathcal{Q} >0$ on $M\setminus {\rm supp}({\rm d}\Phi)$, where $m=n+4t$ for $t\in\mathbb{Z}_{\geq 0}$, $\mathcal{Q}$ is a constant and $\Phi: M\to N$ is a smooth, area non-increasing, spin map which is locally constant near the boundary $\partial M$. Assume that the relative dominant energy condition $\mu-|J|_{g_M} \geq \frac{1}{2} ( {\rm scal}_N \circ \Phi ) \cdot \| \Lambda^2 {\rm d}\Phi \|$ on ${\rm supp}({\rm d}\Phi)$ and the length of the neck
    \begin{equation}
       {\rm dist}_{g_M}({\rm supp}({\rm d}\Phi),\partial M) > \pi \sqrt{\frac{m-1}{2 m \mathcal{Q} }}.
    \end{equation}
Then ${\rm deg}_{\widehat{A}}(\Phi)=0$.  
\end{corollary}

\begin{corollary}\label{cor:IDS-long-neck-relative-to-sphere}
Let $(M, g_M, k)$ be a compact $m$-dimensional spin initial data set with boundary such that $\mu-|J|_{g_M}\geq \mathcal{Q} >0$ on $M\setminus {\rm supp}({\rm d}\Phi)$, where $m\geq 2$ even, $\mathcal{Q}$ is a constant, and $\Phi:M\to S^m$ is a smooth area non-increasing map. Assume that the relative dominant energy condition $\mu-|J|_{g_M} \geq \frac{m(m-1)}{2}$ on ${\rm supp}({\rm d}\Phi)$ and the length of the neck
    \begin{equation}
       {\rm dist}_{g_M}({\rm supp}({\rm d}\Phi),\partial M) > \pi \sqrt{\frac{m-1}{2 m \mathcal{Q} }}.
    \end{equation}
Then ${\rm deg}(\Phi)=0$. 
\end{corollary}
\begin{remark}
When $k \equiv 0$, Corollary \ref{cor:IDS-long-neck-relative-to-sphere} implies Theorem~A in \cite{Cec20}, in fact, whose odd-dimensional case can be obtained by considering the product manifold $M \times S^1$. 
\end{remark}

Secondly, our results concern band width estimates in terms of the lower bound of the dominant energy scalar $\mu-|J|_{g_M}$. For the definition of the bands of infinite vertical $\widehat{A}$-area in Theorem \ref{thm:IDS-band-width-estimate}, we refer to Section \ref{sec:band-width-estimates}.
\begin{theorem}\label{thm:IDS-band-width-estimate}
Let $(M, g_M, k)$ be a compact $m$-dimensional initial data set with boundary such that $\mu-|J|_{g_M}\geq \mathcal{Q} >0$ and $(M, g_M)$ be a spin band of infinite vertical $\widehat{A}$-area, where $m\geq 3$ odd and $\mathcal{Q}$ is a constant.
Suppose that $M$ has a smooth closed hypersurface $\Sigma$ separating $M$ into two parts $M_{-}$ and $M_{+}$ such that ${\rm dist}_{g_M}(p,\partial_{-}M)={\rm dist}_{g_M}(p,\partial_{+}M)$ for every $p\in \Sigma$. Assume that ${\rm tr}_{g_M}k$ has different signs or equals to zero on $M_{\pm}$.
Then
\begin{equation}
{\rm width}(M, g_M, k) \leq 2\pi \sqrt{\frac{m-1}{2m\mathcal{Q}}}.
\end{equation}
\end{theorem}
\begin{remark}
If $k\equiv 0,\mathcal{Q}=\frac{1}{2}m(m-1)$, Theorem \ref{thm:IDS-band-width-estimate} yields Corollary~7.9 in \cite{CZ22}. In particular, it recovers Gromov's torical band inequality with the optimal constant $\frac{2\pi}{m}$ for all dimensions, see \cite[pp. 653]{Gro18}.
\end{remark}

The similar method also gives the following width estimates for $\mathcal{KO}$-bands (see Definition \ref{defn:KO-band}). 
\begin{theorem}\label{thm:IDS-KO-band-width-estimate}
Let $(M, g_M, k)$ be a compact $m$-dimensional initial data set with boundary such that $\mu-|J|_{g_M}\geq \mathcal{Q} >0$ and $(M, g_M)$ be a $\mathcal{KO}$-band, where $m\geq 3$ odd and $\mathcal{Q}$ is a constant. 
Suppose that $M$ has a smooth closed hypersurface $\Sigma$ separating $M$ into two parts $M_{-}$ and $M_{+}$ such that ${\rm dist}_{g_M}(p,\partial_{-}M)={\rm dist}_{g_M}(p,\partial_{+}M)$ for every $p\in \Sigma$. Assume that ${\rm tr}_{g_M}k$ has different signs or equals to zero on $M_{\pm}$.
Then
\begin{equation}
{\rm width}(M, g_M, k) \leq 2\pi \sqrt{\frac{m-1}{2m\mathcal{Q}}}.
\end{equation}
\end{theorem}
\begin{remark}
When $k\equiv 0,\mathcal{Q}=\frac{1}{2}m(m-1)$, Theorem \ref{thm:IDS-KO-band-width-estimate} recovers \cite[Proposition~5.5]{Zei20}.
\end{remark}



Finally, we now turn to our results corresponding to Conjecture \ref{conj:geodesic-collar-neighborhood}. Using Callias operators, we also prove the following result.
\begin{theorem}\label{thm:IDS-width-inequality}
Let $(M,g_M, k)$ be a compact $m$-dimensional spin initial data set with boundary such that the double of $M$ has infinite $\widehat{A}$-area. Let $\mathcal{N}$ be an open geodesic collar neighborhood of the boundary $\partial M$. Assume that $\mu-|J|_{g_M} \geq \mathcal{Q}>0$ on $\mathcal{N}$, where $\mathcal{Q}$ is a constant. Then the width of $\mathcal{N}$ satisfies
\begin{equation}
{\rm width}(\mathcal{N}) \leq  \pi\sqrt{\frac{m-1}{2m\mathcal{Q}}}.
\end{equation}
\end{theorem}
\begin{remark}
Theorem \ref{thm:IDS-width-inequality} with $k\equiv 0, \mathcal{Q}=\frac{1}{2}m(m-1)$ gives an affirmative answer to Conjecture \ref{conj:geodesic-collar-neighborhood}. See also \cite[Theorem~1.7]{CZ22} for a spinorial proof of Conjecture \ref{conj:geodesic-collar-neighborhood} using the information from the mean curvature of the boundary.
\end{remark}

Throughout this paper, all manifolds are assumed to be smooth, oriented and connected.  

The remainder of the paper is organized as follows. In Section \ref{sec:preliminaries}, we provide some backgrounds and facts of technical preparations. In Section \ref{sec:long-neck-principle}, we use the Callias operator approach to supply the proof of Theorem \ref{thm:IDS-long-neck-principle}. In Section \ref{sec:band-width-estimates}, we again employ the Callias operators to prove Theorem \ref{thm:General-IDS-band-width-estimate} and its applications: Theorem \ref{thm:IDS-band-width-estimate} and Theorem \ref{thm:IDS-KO-band-width-estimate}. In Section \ref{sec:width-inequality-of-geodesic-collar-neighborhood}, we firstly apply the Callias operator method to give the proof of Theorem \ref{thm:IDS-width-inequality} similar to that of Theorem \ref{thm:IDS-long-neck-principle}.


\section{Preliminaries}\label{sec:preliminaries}

\vspace{9pt}

In this section, we review some necessary background material and prerequisite results for later use.

\begin{definition}\label{defn:IDS}
An \textit{initial data set} $(M, g_M, k)$ is a Riemannian manifold $(M,g_M)$ with a symmetric $(0,2)$-tensor $k$. Set
\begin{equation}\label{eq:energy-momentum}
\mu=\frac{1}{2}({\rm scal}_M-|k|_{g_M}^2+({\rm tr}_{g_M} k)^2),\quad J={\rm div}_{g_M} k-{\rm d}({\rm tr}_{g_M} k),
\end{equation}
where ${\rm scal}_M$ is the scalar curvature of $(M, g_M)$. Here $\mu$ and $J$ denote the \textit{energy density} and the \textit{momentum density} respectively. An initial data set $(M,g_{M},k)$ is said to satisfy the \textit{dominant energy condition} (DEC) if 
\[
\mu \geq |J|_{g_M}.
\]
\end{definition}
\begin{remark}
An initial data set can be considered as a spacelike hypersurface in a Lorentzian spacetime $(\overline{M},\overline{g})$, where $g_M$ is the induced Riemannian metric on $M$ and $k$ is its second fundamental form $k(X,Y)=\langle \nabla_X^{\overline{M}} e_0, Y \rangle$ with respect to the unit timelike normal vector field $e_0$, where $X,Y\in TM$.
\end{remark}

\begin{remark}
When $k\equiv 0$, the dominant energy condition reduces to the non-negativity of scalar curvature ${\rm scal}_M\geq 0$. When $k=g$, the dominant energy condition reduces to the condition ${\rm scal}_M \geq -m(m-1)$.
\end{remark}

If $\partial M\neq \emptyset$, let $\nu$ be the outward unit normal of $\partial M$ in $M$, then the mean curvature of $\partial M$ with respect to $\nu$ is defined as
\begin{equation}\label{eq:defn-mean-curvature}
    H_{\partial M}:= \frac{1}{m-1} {\rm tr}_{\partial M}(\nabla \nu) = \frac{1}{m-1} \sum_{i=1}^{m-1}\langle e_i, \nabla_{e_i}\nu \rangle.
\end{equation}
where $\nabla \nu$ is the shape operator and $e_1,\cdots,e_{m-1}$ is a local orthonormal frame on $\partial M$.

The reader should note that our notation differs from that commonly used in the literature, in which the mean curvature is the trace of the shape operator $\nabla \nu$ of $\partial M$. We hope that this will not cause the reader any confusion.

\begin{definition}
Given an $m$-dimensional initial data set $(M, g_M, k)$ with boundary $\partial M$, the \textit{outward (inward) null expansion} of $\partial M$ is defined by
\begin{equation}\label{eq:null-expansion}
    \theta_{\partial M}^{\pm}=(m-1) H_{\partial M} \pm {\rm tr}_{\partial M} k.
\end{equation}
\end{definition}

In this section, we describe the Callias operator techniques employed in our main results. 
The method we use is similar to that in \cite{CZ22,CZ21,Liu23}.
We start by considering three exemplary geometric examples which are relevant to Callias operators. 

We firstly introduce the following example of spacetime Gromov-Lawson pairs, an augmentation of a Gromov-Lawson pair \cite[Example~2.5]{CZ22}.
\begin{example}[Spacetime Gromov-Lawson pairs]
Let $(M,g_M,k)$ be a spin initial data set, possibly noncompact and with compact boundary, $\dim M=m$ even. Following \cite{CLZ23}, let $T^*M\oplus \mathbb{R}^2$ be the vector bundle with a metric of signature $(m,2)$, where $\mathbb{R}^2 =\langle y_0,y_1\rangle$ is equipped with the negative definite metric $-{\rm d}y_0^2 - {\rm d}y_1^2$. Let $\mathcal{S}_M \to M$ be the spinor bundle associated to $T^*M\oplus \mathbb{R}^2$ with respect to some representation ${\rm Cl}_{m,2}\hookrightarrow {\rm End}(\Sigma)$. Let $\epsilon_i=c(y_i):\mathcal{S}_M \to \mathcal{S}_M$ be the Clifford multiplication of $y_i$.
We say that a pair of Hermitian bundles $(\mathcal{E},\mathcal{F}) \to M$ with metric connections is a \textit{spacetime Gromov-Lawson pair} on $M$ with \textit{support} $K$ if there exists a compact subset $K$ of the interior $M^{\circ}$ and a parallel unitary bundle isomorphism $\daleth:\left.\mathcal{E}\right|_{M\setminus K} \to \left.\mathcal{F} \right|_{M\setminus K}$ which extends to a smooth bundle map on a neighborhood of $\overline{M\setminus K}$. Then in analogy with Cecchini and Zeidler \cite{CZ22}, we consider the $\mathbb{Z}_2$-graded Dirac bundle
\begin{equation}
S=\mathcal{S}_M \widehat{\otimes} (\mathcal{E} \oplus (\mathcal{F})^{\rm op}),
\end{equation}
where the grading $S=S^+\oplus S^-$ is given by
\begin{equation}
S^+:=(\mathcal{S}_M^+ \otimes \mathcal{E}) \oplus (\mathcal{S}_M^- \otimes \mathcal{F})\quad \text{and} \quad S^-:=(\mathcal{S}_M^+ \otimes \mathcal{F}) \oplus (\mathcal{S}_M^- \otimes \mathcal{E}).
\end{equation}
This bundle carries two involutions $\sigma_0,\sigma_1$ respectively given by
\begin{equation}
\sigma_0:=\epsilon_0 \widehat{\otimes} \begin{pmatrix}
0& \daleth^* \\
\daleth & 0
\end{pmatrix}
: \left. S\right|_{M\setminus K}\to \left. S\right|_{M\setminus K},
\end{equation}
and
\begin{equation}
\sigma_1:=\epsilon_1 \widehat{\otimes} \begin{pmatrix}
0& -\sqrt{-1}\daleth^* \\
\sqrt{-1} \daleth & 0
\end{pmatrix}
: \left. S\right|_{M\setminus K}\to \left. S\right|_{M\setminus K}.
\end{equation}
Let $\slashed{D}_{\mathcal{E}},\slashed{D}_{\mathcal{F}}$ be the twisted Dirac operators on $M$. Note that
\begin{equation}
\slashed{D}_{\mathcal{E}}=
\begin{pmatrix}
0& \slashed{D}_{\mathcal{E}}^{-} \\
\slashed{D}_{\mathcal{E}}^{+} & 0
\end{pmatrix}
\quad \text{and} \quad
\slashed{D}_{\mathcal{F}}=
\begin{pmatrix}
0& \slashed{D}_{\mathcal{F}}^{-} \\
\slashed{D}_{\mathcal{F}}^{+} & 0
\end{pmatrix},
\end{equation}
where $(\slashed{D}_{\mathcal{E}}^{+})^*=\slashed{D}_{\mathcal{E}}^{-}$, $(\slashed{D}_{\mathcal{F}}^{+})^*=\slashed{D}_{\mathcal{F}}^{-}$. Let $\mathcal{D}^{\pm}: C^{\infty}(M,S^{\pm}) \to C^{\infty}(M,S^{\mp})$ are the operators given by
\begin{equation}
\mathcal{D}^{+}=
\begin{pmatrix}
0& \slashed{D}_{\mathcal{F}}^{-} \\
\slashed{D}_{\mathcal{E}}^{+} & 0
\end{pmatrix}
\quad \text{and} \quad
\mathcal{D}^{-}=
\begin{pmatrix}
0& \slashed{D}_{\mathcal{E}}^{-} \\
\slashed{D}_{\mathcal{F}}^{+} & 0
\end{pmatrix}.
\end{equation}
Then Dirac operator $\mathcal{D}: C^{\infty}(M,S) \to C^{\infty}(M,S)$ on $S$ is defined as
\begin{equation}
\mathcal{D}=
\begin{pmatrix}
0& \mathcal{D}^{-} \\
\mathcal{D}^{+} & 0
\end{pmatrix}.
\end{equation}
We have the following Bochner-Lichnerowicz-Schr\"{o}dinger-Weitzenb\"{o}ck (B-L-S-W) formula 
\begin{equation}\label{eq:BLSW}
 \mathcal{D}^2=\nabla^*\nabla+\mathscr{R}
\end{equation}
and the curvature term $\mathscr{R}$ in \eqref{eq:BLSW} is given by 
\begin{equation}\label{eq:zeroterm-GL}
\mathscr{R}=\frac{1}{4} {\rm scal}_M + \mathscr{R}^{\mathcal{E}\oplus\mathcal{F}},
\end{equation}
where $\mathscr{R}^{\mathcal{E}\oplus\mathcal{F}}=\sum_{i<j}c(e^i)c(e^j)({\rm id}_{\mathcal{S}_M}\otimes R_{e_i,e_j}^{\nabla^{\mathcal{E}\oplus\mathcal{F}}})$ is an even endomorphism of the bundle $S$ which depends linearly on the curvature of the connection on $\mathcal{E}\oplus\mathcal{F}$.
\end{example}
Next, we turn to spacetime relative Gromov-Lawson pairs over an initial data set allowed to be non-spin.
\begin{example}[Spacetime relative Gromov-Lawson pair]
Let $(M,g_M,k)$ be an intial data set, possibly noncompact and with compact boundary, $\dim M=m$ even. Following \cite{CLZ23} again, let $T^*M\oplus \mathbb{R}^2$ be the vector bundle with a metric of signature $(m,2)$, where $\mathbb{R}^2 =\langle y_0,y_1\rangle$ is equipped with the negative definite metric $-{\rm d}y_0^2 - {\rm d}y_1^2$. 
Let $\mathcal{S}_M \to M$ be the local spinor bundle associated to $T^*M\oplus \mathbb{R}^2$ with respect to some representation ${\rm Cl}_{m,2}\hookrightarrow {\rm End}(\Sigma)$. Let $\epsilon_i=c(y_i^{\flat}):\mathcal{S}_M \to \mathcal{S}_M$ be the Clifford multiplication of $y_i$, where $\flat$ denotes the isomorphism from $TM$ to $T^*M$.
Let $N$ be a $n$-dimensional Riemannian manifold and $\Phi,\Psi$ be two smooth spin maps from $M$ to $N$. Let $\mathcal{S}_N \to N$ be the local spinor bundle.
Note that, although $\mathcal{S}_M$ and $\mathcal{S}_N$ are locally defined, $\mathcal{S}_M\widehat{\otimes} \Phi^*\mathcal{S}_N$ and $\mathcal{S}_M\widehat{\otimes} \Psi^*\mathcal{S}_N$ are globally defined since $\Phi$ and $\Psi$ are spin maps.
We say that a local pair $(\Phi^*\mathcal{S}_N,\Psi^*\mathcal{S}_N)$ is a \textit{spacetime relative Gromov-Lawson pair} over $M$ with \textit{support} $K$, relative to the background manifold $N$, cf. \cite{Liu23}, if there exists a compact subset $K$ of the interior $M^{\circ}$ and a local parallel unitary bundle isomorphism $\daleth:\Phi^*\mathcal{S}_N \to \Psi^*\mathcal{S}_N$ outside $K$ which extends to a local smooth bundle map over a neighborhood of $\overline{M\setminus K}$. Then we may consider the $\mathbb{Z}_2$-graded Dirac bundle
\begin{equation}
S:=\mathcal{S}_M \widehat{\otimes} (\Phi^*\mathcal{S}_N \oplus (\Psi^*\mathcal{S}_N)^{\rm op}),
\end{equation}
where the grading $S=S^+\oplus S^-$ is given by
\begin{equation}
S^+:=(\mathcal{S}_M^+ \otimes \Phi^*\mathcal{S}_N) \oplus (\mathcal{S}_M^- \otimes \Psi^*\mathcal{S}_N)\quad \text{and} \quad S^-:=(\mathcal{S}_M^+ \otimes \Psi^*\mathcal{S}_N) \oplus (\mathcal{S}_M^- \otimes \Phi^*\mathcal{S}_N).
\end{equation}
This bundle carries two involutions $\sigma_0,\sigma_1$ respectively given by
\begin{equation}
\sigma_0:=\epsilon_0 \widehat{\otimes} \begin{pmatrix}
0& \daleth^* \\
\daleth & 0
\end{pmatrix}
: \left. S\right|_{M\setminus K}\to \left. S\right|_{M\setminus K},
\end{equation}
and
\begin{equation}
\sigma_1:=\epsilon_1 \widehat{\otimes} \begin{pmatrix}
0& -\sqrt{-1}\daleth^* \\
\sqrt{-1} \daleth & 0
\end{pmatrix}
: \left. S\right|_{M\setminus K}\to \left. S\right|_{M\setminus K}.
\end{equation} 
Let $\slashed{D}_{\Phi^*\mathcal{S}_N},\slashed{D}_{\Psi^*\mathcal{S}_N}$ be the twisted Dirac operators on $M$. Note that
\begin{equation}
\slashed{D}_{\Phi^*\mathcal{S}_N}=
\begin{pmatrix}
0& \slashed{D}_{\Phi^*\mathcal{S}_N}^{-} \\
\slashed{D}_{\Phi^*\mathcal{S}_N}^{+} & 0
\end{pmatrix}
\quad \text{and} \quad
\slashed{D}_{\Psi^*\mathcal{S}_N}=
\begin{pmatrix}
0& \slashed{D}_{\Psi^*\mathcal{S}_N}^{-} \\
\slashed{D}_{\Psi^*\mathcal{S}_N}^{+} & 0
\end{pmatrix},
\end{equation}
where $(\slashed{D}_{\Phi^*\mathcal{S}_N}^{+})^*=\slashed{D}_{\Phi^*\mathcal{S}_N}^{-}$, $(\slashed{D}_{\Psi^*\mathcal{S}_N}^{+})^*=\slashed{D}_{\Psi^*\mathcal{S}_N}^{-}$. Let $\mathcal{D}^{\pm}: C^{\infty}(M,S^{\pm}) \to C^{\infty}(M,S^{\mp})$ are the operators is defined as
\begin{equation}
\mathcal{D}^{+}:=
\begin{pmatrix}
0& \slashed{D}_{\Psi^*\mathcal{S}_N}^{-} \\
\slashed{D}_{\Phi^*\mathcal{S}_N}^{+} & 0
\end{pmatrix}
\quad \text{and} \quad
\mathcal{D}^{-}:=
\begin{pmatrix}
0& \slashed{D}_{\Phi^*\mathcal{S}_N}^{-} \\
\slashed{D}_{\Psi^*\mathcal{S}_N}^{+} & 0
\end{pmatrix}.
\end{equation}
Then Dirac operator $\mathcal{D}: C^{\infty}(M,S) \to C^{\infty}(M,S)$ on $S$ is given by
\begin{equation}
\mathcal{D}=
\begin{pmatrix}
0& \mathcal{D}^{-} \\
\mathcal{D}^{+} & 0
\end{pmatrix}.
\end{equation}
The curvature term in the B-L-S-W formula is of the form
\begin{equation}\label{eq:zeroterm-relative-GL}
\mathscr{R}=\frac{1}{4} {\rm scal}_M + \mathscr{R}^{\Phi^*\mathcal{S}_N \oplus \Psi^*\mathcal{S}_N},
\end{equation}
where $\mathscr{R}^{\Phi^*\mathcal{S}_N \oplus \Psi^*\mathcal{S}_N}=\sum_{i<j}c(e^i)c(e^j)({\rm id}_{\mathcal{S}_M}\otimes R_{e_i,e_j}^{\nabla^{\Phi^*\mathcal{S}_N \oplus \Psi^*\mathcal{S}_N}})$.
\end{example} 

Now, fix a complete initial data set $(M,g_M,k)$ with compact boundary $\partial M$ and $S\to M$ a Dirac bundle with support $K$ and with the induced metric connection $\nabla$ and involutions $\sigma_0,\sigma_1$ from any of the above examples, then we define a new connection on $S$ by
\begin{equation}
\widetilde{\nabla}_{e_i}:= \nabla_{e_i}-\frac{1}{2}k_{ij}c(e^j)\sigma_0,
\end{equation}
then the \textit{Dirac-Witten operator} associated to $\widetilde{\nabla}$ is defined as
\begin{equation}
\widetilde{\mathcal{D}}:= c(e^i)\widetilde{\nabla}_{e_i}=\mathcal{D}+\frac{1}{2}({\rm tr}_{g_M}k)\sigma_0.
\end{equation}
Here we have used the Einstein summation convention. Furthermore, for a Lipschitz function $f$ on $M$, we define a modified connection on $S$ by
\begin{equation}
    \widetilde{\nabla}_{e_i}^{f} := \nabla_{e_i} - \frac{1}{2} k_{ij} c(e^j)\sigma_0 - \frac{f}{m}c(e^i)\sigma_1,
\end{equation}
then we consider the associated \textit{Dirac-Witten operator with a Callias potential} (or \textit{Callias operator} for short) that is defined in \cite{CLZ23} by 
\begin{equation}
     \widetilde{\mathcal{D}}^{f} := c(e^i)\widetilde{\nabla}_{e_i}^{f}=\mathcal{D}+\frac{1}{2}({\rm tr}_{g_M}k)\sigma_0 + f\sigma_1 = \widetilde{\mathcal{D}} + f\sigma_1.
\end{equation}
Using the Green-Stokes formula and the B-L-S-W formula, a direct computation then yields the following basic spectral estimate. For any smooth section $u\in C^{\infty}(M,S)$ of $S$ over $M$, we have
\begin{equation}\label{eq:basic-spectral-estimate}
\begin{aligned}
\int_M |\widetilde{\mathcal{D}}^{f} u|^2 dV
= & \int_M |\widetilde{\nabla}^f u|^2 dV
+ \int_M  \frac{1}{4} (({\rm tr}_{g_M}k)^2-|k|_{g_M}^2)|u|^2 dV
+\int_M \langle \mathscr{R} u, u\rangle dV \\
& - \frac{1}{2} \int_M \langle c({\rm div}_{g_M} k - {\rm d}({\rm tr}_{g_M}k))\sigma_0 u, u \rangle dV \\
& +\frac{m-1}{m} \int_M \langle (f^2 + c({\rm d}f)\sigma_1)u, u\rangle dV \\
& - \int_{\partial M} \langle (c(\nu^{\flat})\widetilde{\mathcal{D}}+\widetilde{\nabla}_{\nu}+\frac{m-1}{m}f c(\nu^{\flat})\sigma_1) u, u\rangle dA,
\end{aligned}
\end{equation}
where $\nu$ is the outward unit normal of $\partial M$.
Since $S$ has a $\mathbb{Z}_2$-graded structure, $\widetilde{\mathcal{D}}^{f}$ can be decomposed as $\widetilde{\mathcal{D}}^{f}=\widetilde{\mathcal{D}}^{f,+}\oplus\widetilde{\mathcal{D}}^{f,-}$ where $\widetilde{\mathcal{D}}^{f,\pm}: C^\infty(M,S^{\pm})\to C^\infty(M,S^{\mp})$. Let $s:\partial M\to \{\pm 1\}$ be the choice of signs, $\flat$ be the isomorphism from $TM$ to $T^*M$. Together with $\sigma_1$ this defines the following \textit{boundary chirality operator}
\begin{equation}
 \chi:= sc(\nu^{\flat})\sigma_1 : \left. S\right|_{\partial M} \to \left. S\right|_{\partial M}.
\end{equation}
We now consider an elliptic boundary value problem 
\begin{equation}\label{eq:boundary-value-problem}
\begin{cases}
\widetilde{\mathcal{D}}^{f} u=0\quad & \text{in}\ M, \\
\chi \left.u\right|_{\partial M}= \left.u\right|_{\partial M} \quad & \text{on}\ \partial M.
\end{cases}
\end{equation}
Let $\widetilde{\mathcal{D}}^{f,s}$ denote the operator $\widetilde{\mathcal{D}}^{f}$ on the domain
\begin{equation}\label{eq:domain}
    C_{\sigma_1,s}^\infty(M,S):=\left\{u\in C^\infty(M,S) \Big| \, \left. \chi u\right|_{\partial M}=\left.u\right|_{\partial M}\right\}.
\end{equation}
It follows from \cite[Theorem~3.4]{CZ22} that $\widetilde{\mathcal{D}}^{f,s}$ is self-adjoint and Fredholm.
Since $\chi$ is even with respect to the grading on $S$, the operator $\widetilde{\mathcal{D}}^{f,s}$ can also be decomposed as $\widetilde{\mathcal{D}}^{f,s}=\widetilde{\mathcal{D}}^{f,s,+}\oplus \widetilde{\mathcal{D}}^{f,s,-}$ where $\widetilde{\mathcal{D}}^{f,s,\pm}$ are adjoint to each other. Therefore, the \textit{index} of $\widetilde{\mathcal{D}}^{f,s}$ is defined as
\begin{equation}\label{eq:defn-of-index}
{\rm ind}(\widetilde{\mathcal{D}}^{f,s}):=\dim(\ker(\widetilde{\mathcal{D}}^{f,s,+}))-\dim(\ker(\widetilde{\mathcal{D}}^{f,s,-})).
\end{equation}
By \cite[Lemma~3.8]{CZ22}, ${\rm ind}(\widetilde{\mathcal{D}}^{f,s})$ does not depend on the choice of Callias potentials.

We shall calculate the boundary term $c(\nu^{\flat})\widetilde{\mathcal{D}}+\widetilde{\nabla}_{\nu}$ under the boundary condition in \eqref{eq:boundary-value-problem}. In order to achieve this, we consider the following boundary Dirac operator $\mathcal{A}$ on the boundary Dirac bundle $\left. S\right|_{\partial M}$
\begin{equation}
   \mathcal{A}:= \sum_{i=1}^{m-1} c^{\partial}(e^i)\nabla_{e_i}^{\partial},
\end{equation}
where $c^{\partial}(e^i)=-c(e^i)c(\nu^{\flat})$ and $\nabla_{e_i}^{\partial} =\nabla_{e_i} - \frac{1}{2} c^{\partial} (\nabla_{e_i}\nu^{\flat})$ is the Clifford multiplication and connection respectively. Note that
\begin{equation}\label{eq:boundary-Dirac-operator-identity}
\begin{aligned}
\mathcal{A} =& \frac{1}{2}(m-1) H_{\partial M} + c(\nu^{\flat}) \mathcal{D} + \nabla_{\nu}\\
=& \frac{1}{2}(m-1)H_{\partial M} + c(\nu^{\flat})(\widetilde{\mathcal{D}}-\frac{1}{2}({\rm tr}_{g_M}k) \sigma_0 ) + \widetilde{\nabla}_{\nu} +\frac{1}{2} k_{j\nu} c(e^j) \sigma_0 \\
=& \frac{1}{2}(m-1) H_{\partial M}  - \frac{1}{2}({\rm tr}_{g_M}k) c(\nu^{\flat}) \sigma_0 + \frac{1}{2} k_{\nu\nu}c(\nu^{\flat})\sigma_0 + \frac{1}{2} k_{a\nu} c(e^a)\sigma_0 \\
&+ c(\nu^{\flat})\widetilde{\mathcal{D}}+\widetilde{\nabla}_{\nu} \\
=& \frac{1}{2}(m-1) H_{\partial M} - \frac{1}{2}({\rm tr}_{\partial M}k) c(\nu^{\flat}) \sigma_0 + \frac{1}{2} k_{a\nu} c(e^a)\sigma_0 + c(\nu^{\flat})\widetilde{\mathcal{D}}+\widetilde{\nabla}_{\nu}.
\end{aligned}
\end{equation}
Note that, the boundary condition in \eqref{eq:boundary-value-problem} implies that
\begin{equation}\label{eq:boundary-condition-result}
\langle \mathcal{A}u, u\rangle=\langle c(\nu^{\flat})\sigma_0 u, u\rangle=0.
\end{equation}
Here we have used the fact that $\chi^2={\rm Id}, \chi\mathcal{A}=-\mathcal{A}\chi$, and $\chi(c(\nu^{\flat})\sigma_0)=-(c(\nu^{\flat})\sigma_0)\chi$. We also note that
\begin{equation}\label{eq:k-top-boundary-term}
\begin{aligned}
\langle k_{a\nu} c(e^a) \sigma_0 u, u \rangle  & = k_{\nu\nu}\langle c(\nu^{\flat})\sigma_0 u, u\rangle + \langle \sum_{a=1}^{m-1} k_{a\nu}c(e^a)\sigma_0 u, u \rangle \\
&\geq - (\sum_{a=1}^{m-1} k_{a\nu}^2)^{\frac{1}{2}} |u|^2 = -|k(\nu,\cdot)^{\top}|\ |u|^2.
\end{aligned}
\end{equation}
Let $u\in C^{\infty}(M,S;\chi)$ be a smooth section of $S$ over $M$ satisfying the boundary condition in \eqref{eq:boundary-value-problem}.
Therefore, combining \eqref{eq:boundary-Dirac-operator-identity}, \eqref{eq:boundary-condition-result}, and \eqref{eq:k-top-boundary-term}, we have 
\begin{equation}\label{eq:key-part-of-boundary-term}
\begin{aligned}
&-\int_{\partial M} \langle (c(\nu^{\flat})\widetilde{\mathcal{D}}+\widetilde{\nabla}_{\nu})u, u\rangle dA\\
= & \frac{1}{2} \int_{\partial M} ((m-1)H_{\partial M} - |k(\nu,\cdot)^{\top}|) |u|^2 dA.
\end{aligned}
\end{equation}

To summarize, from \eqref{eq:basic-spectral-estimate}, \eqref{eq:key-part-of-boundary-term}, and the inequality $\langle c({\rm d}f)\sigma_1 u, u\rangle\leq |{\rm d}f||u|^2$, we arrive at the following proposition.
\begin{proposition}
With the notation above, we have
\begin{equation}\label{spectralestimate}
\begin{aligned}
\int_M |\widetilde{\mathcal{D}}^{f,s} u|^2 dV \geq & \int_M |\widetilde{\nabla}^f u|^2 dV + \int_M  \frac{1}{4} (({\rm tr}_{g_M}k)^2-|k|_{g_M}^2)|u|^2 dV
+\int_M \langle \mathscr{R} u, u\rangle dV \\
& - \frac{1}{2} \int_M \langle c({\rm div}_{g_M} k - {\rm d}({\rm tr}_{g_M}k))\sigma_0 u, u \rangle dV \\
& +\frac{m-1}{m} \int_M (f^2 - |{\rm d}f| )|u|^2 dV\\
& + \int_{\partial M} \left(\frac{1}{2}(m-1)H_{\partial M} -\frac{1}{2} |k(\nu,\cdot)^{\top}| - \frac{m-1}{m} fs \right)|u|^2 dA
\end{aligned}
\end{equation}
for any $u\in C^{\infty}(M,S;\chi)$. 
\end{proposition}
Finally, in the remainder of this section, we review the example of odd bands in \cite{CZ22} for the sake of completeness.
\begin{example}[Odd bands]\label{example:odd-bands}
Let $(M,g_M)$ is an odd-dimensional Riemannian spin band.
Let $\mathcal{S}_M\to M$ be the spinor bundle with spin connection and Clifford multiplication $c_{\mathcal{S}}$. Let $\mathcal{E}\to M$ be a Hermitian bundle with a metric connection. Then $S:=(\mathcal{S}_M \otimes \mathcal{E}) \oplus (\mathcal{S}_M\otimes \mathcal{E})$ becomes a $\mathbb{Z}_2$-graded Dirac bundle with the Clifford multiplication
\begin{equation}
c(\xi):=\begin{pmatrix}
0& c_{\mathcal{S}}(\xi) \otimes {\rm id}_{\mathcal{E}} \\
c_{\mathcal{S}}(\xi) \otimes {\rm id}_{\mathcal{E}} & 0
\end{pmatrix}
\end{equation}
and involution
\begin{equation}
\sigma_1:=\begin{pmatrix}
0& -\sqrt{-1} \\
\sqrt{-1} & 0
\end{pmatrix}.
\end{equation}
Let $\slashed{D}_{\mathcal{E}}$ be the twisted Dirac operator on $M$, then Dirac operator on $S$ is defined as 
\begin{equation}
\mathcal{D}=\begin{pmatrix}
0& \slashed{D}_{\mathcal{E}} \\
\slashed{D}_{\mathcal{E}} & 0
\end{pmatrix}.
\end{equation}
The curvature term in the B-L-S-W formula is of the form
\begin{equation}\label{eq:zeroterm-odd-bands}
\mathscr{R}=\frac{1}{4} {\rm scal}_M + \mathscr{R}^{\mathcal{E}},
\end{equation}
where $\mathscr{R}^{\mathcal{E}}=\sum_{i<j}c(e^i)c(e^j)({\rm id}_{\mathcal{S}_M\oplus \mathcal{S}_M} \otimes R_{e_i,e_j}^{\nabla^{\mathcal{E}}})$.
\end{example}
Let $(M,g_M,k)$ be an odd-dimensional compact initial data set with boundary. Furthermore, assume that $(M,g_M)$ is a spin band in Example \ref{example:odd-bands}. In this case, for any Lipschitz potential function $f$ we consider the Callias operator on $S$ given by
\begin{equation}
\widetilde{\mathcal{D}}^{f}:=\mathcal{D}+\frac{1}{2}({\rm tr}_{g_M}k)\sigma_1+f\sigma_1.
\end{equation}

As before, a routine calculation procedure that we omit here allows us to characterize a spectral estimate of the Callias operator.
\begin{proposition}\label{pro:odd-band-spectral-estimates}
Let $(M,g_M)$ be an odd band, $s: \partial M\to \{\pm 1\}$ be the choice of signs, and $\widetilde{\mathcal{D}}^{f,s}$ be a Callias operator. Then for every smooth section $u\in C^\infty(M,S;\chi)$,
\begin{equation}\label{eq:odd-bands-spectral-estimate}
\begin{aligned}
 \int_M |\widetilde{\mathcal{D}}^{f,s} u|^2 dV \geq &  \int_M \left( | \widetilde{\nabla}^f u |^2  +  \frac{1}{2} (\mu-|J|_{g_M})|u|^2 + \langle \mathscr{R}^{\mathcal{E}} u, u \rangle \right) dV\\
& + \frac{m-1}{m} \int_M  (f^2-|{\rm d} f|+f ({\rm tr}_{g_M}k)) |u|^2 dV \\
& +  \int_{\partial M}  \left( \frac{1}{2}(m-1)H_{\partial M}  + \left(\frac{1}{2}({\rm tr}_{\partial M} k) - \frac{m-1}{m} f \right) s \right) |u|^2 dA.
\end{aligned}
\end{equation}
\end{proposition}


\section{A long neck principle under the relative dominant energy condition}\label{sec:long-neck-principle}
In this section, we aim to prove a more general long neck principle in the case where the map is from an initial data set to a Riemannian manifold with non-negative curvature operator. 
\begin{definition}[\cite{GS02}]\label{defn:curvature-operator}
The \textit{curvature operator} $\mathcal{R}_N:\Lambda^2TN\to \Lambda^2TN$ is a self-adjoint operator induced by the Riemannian curvature tensor $R$ of Riemannian manifold $(N,g_N)$, such that
\begin{equation}
g_N(\mathcal{R}_N(e_i\wedge e_j),e_k\wedge e_l)=-g_N(R_{e_i,e_j}e_k,e_l).
\end{equation}
where $e_1,\cdots,e_n$ is a locally orthonormal frame of $TM$.
\end{definition}
Since $\mathcal{R}_N$ is self-adjoint it makes sense to talk about the positivity of $\mathcal{R}_N$. When $\mathcal{R}_N\geq 0$, $g_N$ has non-negative sectional curvature. 
\begin{definition}[\cite{GL83,Lis10}]
Let $(M,g_M)$ and $(N,g_N)$ be two Riemannian manifolds and $\Phi: M\to N$ be a smooth map. Then $\Phi$ is said to be \textit{area non-increasing}, if for all $p\in M$, we have $\|{\rm d}\Phi(v)\wedge {\rm d}\Phi(w)\|_{g_N}\leq \|v\wedge w\|_{g_M}$ for any $v,w\in T_pM$. The \textit{area contraction constant} at $p\in M$ is defined to be the operator norm of the induced map $\Lambda^2{\rm d}\Phi:\Lambda^2T_pM\to \Lambda^2T_{\Phi(p)}N$ on $2$-vectors, i.e.,
\begin{equation}
    \|\Lambda^2{\rm d}\Phi\|:M\to[0,\infty),\quad p\mapsto \max_{0\not=\eta\in \Lambda^2T_pM}\frac{\|\Lambda^2{\rm d}\Phi(\eta)\|_{g_N}}{\|\eta\|_{g_M}}.
\end{equation}
\end{definition}
Let $\|\Lambda^2{\rm d}\Phi\|_{\infty}=\sup\limits_{p\in M}\|\Lambda^2{\rm d}\Phi\|(p)$. Clearly, $\Phi$ is area non-increasing if and only if $\|\Lambda^2{\rm d}\Phi\|_{\infty}\leq 1$.

\begin{definition}[\cite{GL83}]
Let $(M,g_M)$ and $(N,g_N)$ be two compact Riemannian manifolds such that $\dim M-\dim N=4t$ where $t\in\mathbb{Z}_{\geq 0}$. Let $\omega$ be a volume form on $N$ with non-zero integral and let $\widehat{A}(M)$ be the total $\widehat{A}$-class of $M$. Then the \textit{$\widehat{A}$-degree} of the smooth map $\Phi: M\to N$ is given by
\begin{equation}
{\rm deg}_{\widehat{A}}(\Phi)=\dfrac{\int_M\Phi^*\omega\wedge\widehat{A}(M)}{\int_N\omega}.
\end{equation}
\end{definition}
If $\dim M=\dim N$, then ${\rm deg}_{\widehat{A}}(\Phi)={\rm deg}(\Phi)$. If $M$ is a compact manifold with non-empty boundary, $N$ is closed and $\Phi$ is locally constant near the boundary $\partial M$, then $\Phi$ has a well-defined $\widehat{A}$-degree.

\begin{definition}[\cite{GL83,GS02}]
A smooth map $\Phi: M\to N$ between two oriented manifolds is a \textit{spin map} if it is compatible with the second Stiefel-Whitney classes, i.e., 
\begin{equation}\label{eq:spin-map}
    w_2(TM)=\Phi^* w_2(TN).
\end{equation}
\end{definition}
Therefore, in case $N$ is spin, $\Phi$ is a spin if and only if $M$ is spin. 
For the presentation of our result of long neck problem for initial data sets, we make the following definition:
\begin{definition}\label{defn:relative-DEC}
For a spin map $\Phi: M\to N$  from an initial data set $(M,g_M, k)$ to a Riemannian manifold $(N,g_N)$, $(M, g_{M}, k)$ is said to satisfy a \textit{relative dominant energy condition}, if
\[
  \mu-|J|_{g_M} \geq \frac{1}{2}({\rm scal}_{N}\circ \Phi) \cdot \|\Lambda^2{\rm d}\Phi\|
\]
on $M$.
\end{definition}
Using the same argument as in \cite[Lemma~5.1]{CZ22} and also \cite[Lemma~3.1]{Liu23}, we thus obtain the following lemma for the spacetime relative Gromov-Lawson pair over initial data sets.
\begin{lemma}\label{lem:relative-GL-pair}
Let $(N, g_N)$ be a compact $n$-dimensional Riemannian manifold with non-negative curvature operator $\mathcal{R}_N\geq 0$ on $\Lambda^2TN$ and non-vanishing Euler characteristic, where $n\geq 2$ even. Let $(M,g_M, k)$ be a compact $m$-dimensional spin initial data set with non-empty boundary, where $m=n+4t$ for $t\in\mathbb{Z}_{\geq 0}$. Let $\Phi: M\to N$ be a smooth spin map which is locally constant near the boundary $\partial M$ and of non-zero $\widehat{A}$-degree. If the length of the neck ${\rm dist}_{g_M}({\rm supp}({\rm d}\Phi),\partial M)\geq \ell$ for some number $\ell$, then there exists a spacetime relative Gromov-Lawson pair $(\mathcal{\mathcal{E}},\mathcal{F})$ such that 
\begin{itemize}
\item[(i)] $\mathscr{R}_p^{\mathcal{E}\oplus \mathcal{F}}\geq -\frac{1}{4}({\rm scal}_{N}\circ \Phi)\cdot \|\Lambda^2{\rm d}\Phi\|(p)$ for all $p\in M$;
\item[(ii)] $(\mathcal{E},\mathcal{F})$ has support $K:=\{p\in M|{\rm dist}_{g_M}(p,\partial M) \geq \ell\}\supseteq {\rm supp}({\rm d}\Phi)$;
\item[(iii)] ${\rm indrel}(M;\mathcal{E},\mathcal{F})\neq 0$.
\end{itemize}
\end{lemma}
With this lemma, we are now ready to deduce Theorem \ref{thm:IDS-long-neck-principle}.
\begin{proof}[Proof of Theorem \ref{thm:IDS-long-neck-principle}]
Assume, by contradiction, that ${\rm deg}_{\widehat{A}}(\Phi)\neq 0$. Moreover, restating the hypotheses of the theorem, we have 
\begin{itemize}
\item[(iv)] 
 $\mu - |J|_{g_M} \geq \frac{1}{2} ({\rm scal}_N\circ \Phi)\cdot \|\Lambda^2{\rm d}\Phi\|$ on ${\rm supp}({\rm d}\Phi)$; 
\item[(v)] there exists an $\varepsilon>0$ such that 
\[
{\rm dist}_{g_M}({\rm supp}({\rm d}\Phi),\partial M) \geq \omega:=\pi\sqrt{\frac{m-1}{2m\mathcal{Q}}}+\varepsilon.
\]
\end{itemize}
Then we choose a spacetime relative Gromov-Lawson pair $(\mathcal{E},\mathcal{F})$ satisfying the conditions (i) to (iii) from Lemma \ref{lem:relative-GL-pair}.

We shall define a sequence of bounded Lipschitz functions $f_j$ on $M$, which satisfies a certain differential inequality and have the property that $f_j\to +\infty$ on $\partial M$ as $j\to +\infty$, in the following way as in \cite{HKKZ23} (see also \cite{Liu23}). Let $r(p)={\rm dist}_{g_M}(p,\partial M)$. Then for each $j$ we consider
\begin{equation}\label{eq:defn-potential}
\begin{aligned}
f_j(p)=\begin{cases}
\frac{\pi}{2\omega}\cot\left(\frac{\pi}{2\omega}r(p)+\frac{1}{j}\right)\quad &\text{if}\ r(p)\leq \frac{2\omega}{\pi}(\frac{\pi}{2}-\frac{1}{j}), \\
0\quad &\text{otherwise}.
\end{cases}
\end{aligned}
\end{equation}
Since ${\rm dist}_{g_M}({\rm supp}({\rm d}\Phi),\partial M)\geq \omega$, we have that $\left.f_j\right|_{{\rm supp}({\rm d}\Phi)}=0$. Fix a compact subset $\Omega\subset M^{\circ}$ such that for all sufficiently large $j$ we have
\begin{subequations}
\begin{align}
\frac{3}{2}f_j^2-|{\rm d} f_j|\geq 1\quad &\text{on}\ M\setminus\Omega,\label{eq:characteristic-1}\\
\frac{m-1}{m}(f_j^2-|{\rm d} f_j|)+\frac{\mathcal{Q}}{2}\geq C_{\varepsilon}\quad &\text{on}\ M, \label{eq:characteristic-2}
\end{align}
\end{subequations}
where $C_{\varepsilon}>0$ depends on $\varepsilon,m$ and $\mathcal{Q}$. From \cite[Lemma~2.1]{Liu23} and (iii) in Lemma \ref{lem:relative-GL-pair}, the corresponding Callias operator subject to the sign $s=-1$ satisfies
\begin{equation}\label{eq:non-vanishing-of-index}
{\rm ind}(\widetilde{\mathcal{D}}^{f_j,-1})={\rm indrel}(M;\mathcal{E},\mathcal{F})\neq 0.
\end{equation}
In particular, we may obtain a non-trivial element $u_j\in\ker(\widetilde{\mathcal{D}}^{f_j,-1})$ for each $j$. According to \eqref{eq:energy-momentum}, \eqref{eq:zeroterm-relative-GL}, \eqref{spectralestimate}, and Cauchy-Schwarz inequality, we have
\begin{equation}\label{eq:3.13}
\begin{aligned}
& -\int_{\partial M} \left(\frac{1}{2}(m-1)H_{\partial M} -\frac{1}{2} |k(\nu,\cdot)^{\top}| + \frac{m-1}{m} f_j \right)|u_j|^2 dA\\
\geq & \int_M |\widetilde{\nabla}^{f_j} u_j|^2 dV + \frac{1}{2} \int_M (\mu - |J|_{g_M})|u_j|^2 dV +\int_M \langle \mathscr{R}^{\mathcal{E}\oplus \mathcal{F}} u_j, u_j\rangle dV\\
& +\frac{m-1}{m} \int_M (f_j^2 - |{\rm d}f_j| )|u_j|^2 dV.
\end{aligned}
\end{equation}
Hence the boundary term of \eqref{eq:3.13} is non-positive when $j\to +\infty$. Similar to \cite{HKKZ23} (also \cite{Liu23}), using the following relation  
\begin{equation}\label{eq:boundary-identity}
\begin{aligned}
\int_{\partial M}\langle c(\nu^{\flat}) u_j, f_j\sigma_1 u_j\rangle dA 
&=-\int_{M}\left(2f_j^2|u_j|^2+\langle u_j, c({\rm d} f_j)\sigma_1 u_j\rangle\right) dV,
\end{aligned}
\end{equation}
the inequality \eqref{eq:characteristic-1} together with Cauchy-Schwarz inequality, the boundary condition of \eqref{eq:boundary-value-problem} and the property of $\left. f_j\right|_{\partial M}$, we have
\begin{equation}\label{eq:3.18}
\begin{aligned}
\int_{M\setminus\Omega}\left(\frac{1}{2}f_j^2+1\right)|u_j|^2 dV & \leq \int_{M\setminus \Omega}(2f_j^2-| {\rm d} f_j|)|u_j|^2 dV \\
& \leq\int_{\Omega}(|{\rm d} f_j|-2f_j^2)|u_j|^2 dV.
\end{aligned}
\end{equation}
Note that $\max_{\Omega}|u_j|\neq 0$, otherwise this estimate implies that $u_j$ vanishes globally. Therefore we may assume that $\max_{\Omega}|u_j|=1$ by appropriate rescaling. From \eqref{eq:characteristic-2} and \eqref{eq:3.18},
\begin{equation}\label{eq:bound-of-spinor}
\int_{\Omega}|u_j|^2 dV + \int_{M\setminus\Omega}\left(\frac{1}{2}f_j^2+1\right)|u_j|^2 dV \leq \left(\frac{m\mathcal{Q}}{2(m-1)}+1\right)|\Omega|.
\end{equation}
Let 
\[
  \Upsilon_j=\min\limits_{\partial M}\left(\frac{m-1}{m} f_j-\frac{1}{2}\left| (m-1)H_{\partial M} - |k(\nu,\cdot)^{\top}| \right| \right). 
\]
Then $\Upsilon_j\to +\infty$ as $j\to +\infty$. Applying \eqref{eq:3.13}, \eqref{eq:characteristic-2} and (i) in Lemma \ref{lem:relative-GL-pair}, we have 
\begin{equation}\label{eq:bound-of-first-derivative}
\begin{aligned}
&\int_{M}\big|\widetilde{\nabla}^{f_j} u_j \big|^2 dV +\int_{\partial M}\Upsilon_j|u_j|^2 dA \\
\leq & \int_{M}\big|\widetilde{\nabla}^{f_j} u_j\big|^2 dV + \int_{\partial M} \left(\frac{1}{2}(m-1)H_{\partial M} -\frac{1}{2} |k(\nu,\cdot)^{\top}| + \frac{m-1}{m} f_j \right)|u_j|^2 dA\\
\leq & - \frac{m-1}{m} \int_{M} (f_j^2-|{\rm d} f_j|) |u_j|^2 dV -\int_{M} \left( \frac{1}{2} (\mu-|J|_{g_M})|u_j|^2 +\langle \mathscr{R}^{\mathcal{E}\oplus\mathcal{F}} u_j, u_j \rangle \right) dV \\
\leq &\int_{M} \left( \frac{\mathcal{Q}}{2} -C_{\varepsilon} \right) |u_j|^2 dV + \int_{M} \left(\frac{1}{2} \left| \mu-|J|_{g_M} \right| + \| \mathscr{R}^{\mathcal{E}\oplus\mathcal{F}} \|_{\infty}  \right) |u_j|^2 dV \\ 
\leq &  C_1
\end{aligned}
\end{equation}
for some constant $C_1$ independent of $j$. From \eqref{eq:bound-of-spinor} and \eqref{eq:bound-of-first-derivative}, the sequence $u_j$ is uniformly bounded in $H^1(M)$. Thus $u_j$ weakly subconverges to a function $u$ in $H^1(M)$ with strong convergence in $H^s(M)$ for any $s\in [\frac{1}{2},1)$.
Since the trace map $\tau: H^s(M)\to H^{s-\frac{1}{2}}(\partial M)$ is continuous,
$u_j$ converges subsequentially to $\tau(u)$ in $L^2(\partial M)$. However, since $\Upsilon_j\to +\infty$ we find that \eqref{eq:bound-of-first-derivative} yields $\tau(u)=0$ on $\partial M$, and hence $u\in H_0^1(M^{\circ})$. 
Let $U=\left\{p\in M\big|{\rm d}_p\Phi\neq 0\right\}$ and $U_{\delta}=\left\{p\in U\big| \, \|\Lambda^2{\rm d}\Phi\|(p)<\delta\right\}$ for suitable $0<\delta<1$. Then using our spectral estimate \eqref{eq:3.13}, $\mathscr{R}^{\mathcal{E}\oplus\mathcal{F}}=0$ on $M\setminus U$ together with \eqref{eq:characteristic-2}, we obtain
\begin{equation}\label{eq:estimateinequality}
\begin{aligned}
& -\int_{\partial M} \left(\frac{1}{2}(m-1)H_{\partial M} -\frac{1}{2} |k(\nu,\cdot)^{\top}| + \frac{m-1}{m} f_j \right)|u_j|^2 dA \\
\geq &  \int_{U} \left( \frac{1}{2} (\mu-|J|_{g_M}) |u_j|^2 + \langle \mathscr{R}^{\mathcal{E}\oplus\mathcal{F} } u_j , u_j \rangle  \right) dV \\
& +  \int_{M\setminus U} \left( \frac{m-1}{m}(f_j^2-|{\rm d} f_j|) + \frac{\mathcal{Q}}{2} \right)  |u_j|^2 dV \\
\geq & \int_{U}  \left( \frac{1}{2} (\mu-|J|_{g_M}) - \frac{1}{4} ({\rm scal}_N\circ \Phi) \cdot \|\Lambda^2{\rm d}\Phi\| (p) \right) |u_j|^2 dV +  \int_{M\setminus U} C_{\varepsilon} |u_j|^2 dV \\
\geq & \int_{\Omega\cap U_{\delta}} ( \frac{1}{2}(\mu-|J|_{g_M}) - C_{\delta} ) |u_j|^2 dV + \int_{\Omega \setminus U} C_{\varepsilon} |u_j|^2 dV.
\end{aligned}
\end{equation}
Here we have chosen a sufficiently small $\delta$ such that $C_{\delta}=\frac{1}{4}({\rm scal}_N\circ \Phi)\cdot \|\Lambda^2{\rm d}\Phi\|(p)<\frac{1}{2}(\mu-|J|_{g_M})$ for all $p\in \Omega \cap U_{\delta}$. Furthermore, taking the limit as $j\to +\infty$ in \eqref{eq:estimateinequality} with weak lower semi-continuity of the $H^1$-norm, Fatou's lemma, and strong convergence in $L^2$, we get
\begin{equation}
0 \geq \int_{\Omega\cap U_{\delta}} ( \frac{1}{2}(\mu-|J|_{g_M}) - C_{\delta} ) |u|^2 dV + \int_{\Omega \setminus U} C_{\varepsilon} |u|^2 dV.
\end{equation}
We conclude that the last two integrands are positive since $\max_{\Omega}|u|=1$. Thus we arrive at a contradiction and hence complete the proof of Theorem \ref{thm:IDS-long-neck-principle}. 
\end{proof}


\section{Band width estimates for initial data sets}\label{sec:band-width-estimates}

Now we will turn our attention to the study of bands. We begin by reviewing the concept of Gromov's bands \cite{Gro18} to deal with our results.
\begin{definition}\label{defn:band}
A \textit{band} is a compact manifold $M$ with boundary $\partial M=\partial_{-}M\sqcup \partial_{+}M$. The \textit{width} of a Riemannian band $(M, g_M)$, denoted by ${\rm width}(M, g_M)$, is the distance between $\partial_{-}M$ and $\partial_{+}M$ with respect to $g_M$.
\end{definition}

The following relevant definitions and examples of bands of infinite vertical $\widehat{A}$-area and $\mathcal{KO}$-bands can be found in \cite{CZ22} and \cite{Zei20,Zei22} respectively.
\begin{definition}\label{defn:infinite-vertical-Ahatarea}
A band $M$ is said to have \textit{infinite vertical $\widehat{A}$-area}, if for every $\epsilon>0$, there exists a Hermitian vector bundle $\mathcal{E}\to M$ such that $\|R^\mathcal{E}\|_{\infty}<\epsilon$ and such that we have
\begin{equation}\label{eq:infinite-vertical-Ahatarea}
\int_{\partial_{-}M} \widehat{A}(\partial_{-}M) \wedge {\rm ch}(\left.\mathcal{E}\right|_{\partial_{-}M}) \neq 0.
\end{equation}
\end{definition}

\begin{example}
All overtorical bands and $\widehat{A}$-overtorical bands are of infinite vertical $\widehat{A}$-area. On the other hand, $M=N\times [-1,1]$ is a band of infinite vertical $\widehat{A}$-area, where $N$ is a closed even-dimensional spin manifold with infinite $\widehat{A}$-area.
\end{example}

\begin{definition}\label{defn:KO-band}
A band $M$ is called a \textit{$\mathcal{KO}$-band} if $M$ is spin and admits a flat bundle $\mathcal{E}\to M$ of finitely generated projective Hilbert $\mathcal{A}$-modules for some unital Real ${\rm C}^*$-algebra $\mathcal{A}$ such that the twisted Dirac operator on $\partial_{\pm} M$ has non-vanishing index ${\rm ind}(\slashed{D}_{\partial_{-}M,\left.\mathcal{E}\right|_{\partial_{-}M}})\neq 0\in {\rm KO}_{n-1}(\mathcal{A})$.
\end{definition}

\begin{example}
All overtorical bands and $\widehat{A}$-overtorical bands are $\mathcal{KO}$-bands. On the other hand, $M=N\times [-1,1]$ is a $\mathcal{KO}$-band, where $N$ is a closed spin manifold with Rosenberg index $\alpha(N)\neq 0$.
\end{example}

\begin{theorem}\label{thm:General-IDS-band-width-estimate}
Let $(M, g_M, k)$ be a compact $m$-dimensional initial data set with boundary such that $\mu-|J|_{g_M}\geq \mathcal{Q} >0$ and $(M, g_M)$ be a Riemannian band, where $m\geq 3$ odd and $\mathcal{Q}$ is a constant. 
Suppose that $M$ has a smooth closed hypersurface $\Sigma$ separating $M$ into two parts $M_{-}$ and $M_{+}$ such that ${\rm dist}_{g_M}(p,\partial_{-}M)={\rm dist}_{g_M}(p,\partial_{+}M)$ for every $p\in\Sigma$. Assume that ${\rm tr}_{g_M}k$ has different signs or equals to zero on $M_{\pm}$.
For any $\epsilon>0$, if there exists a Hermitian bundle $\mathcal{E}$ such that $\|\mathscr{R}^{\mathcal{E}}\|_{\infty}<\epsilon$ and the kernel of the associated Callias operator $\widetilde{\mathcal{D}}^{f,s}$ subject to $s(\partial_{\pm} M)=\mp 1$ does not vanish, then
\begin{equation}
{\rm width}(M, g_M, k) \leq 2\pi \sqrt{\frac{m-1}{2m\mathcal{Q}}}.
\end{equation}
\end{theorem}

\begin{proof}[Proof of Theorem \ref{thm:General-IDS-band-width-estimate}]
Assume, by contradiction, that there exists an $\varepsilon>0$ such that 
\[
  {\rm width}(M, g_M) \geq \omega:= 2\pi \sqrt{\frac{m-1}{2m\mathcal{Q}}} + \varepsilon.
\]
We define a sequence of functions $f_{\varepsilon, j}$ on $M$ in the following. Let $r_{\pm}(p)={\rm dist}_{g_M}(p,\partial_{\pm}M)$, then for each $j$ we consider
\begin{equation}\label{eq:IDS-band-width-defn-potential}
\begin{aligned}
f_{\varepsilon, j}(p)=\begin{cases}
-\frac{\pi}{\omega}\cot\left(\frac{\pi}{\omega}r_{-}(p)+\frac{1}{j}\right)\quad &\text{if}\ r_{-}(p)\leq \frac{\omega}{\pi} ( \frac{\pi}{2} -\frac{1}{j} ),\\
\frac{\pi}{\omega}\cot\left(\frac{\pi}{\omega}r_{+}(p)+\frac{1}{j}\right)\quad &\text{if}\ r_{+}(p)\leq \frac{\omega}{\pi} ( \frac{\pi}{2} -\frac{1}{j} ),\\
0\quad &\text{otherwise}.
\end{cases}
\end{aligned}
\end{equation}
Let $K_{\varepsilon}$ denote the region defined by the third case in \eqref{eq:IDS-band-width-defn-potential}.
Note that $f_j$ is bounded Lipschitz on $M$ and $f_{\varepsilon, j} \to \pm \infty$ on $\partial_{\pm} M$ when $j\to + \infty$. Fix a compact subset $\Omega\subset M^{\circ}$ such that for all sufficiently large $j$ we have 
\begin{subequations}
\begin{align}
\frac{3}{2}f_{\varepsilon, j}^2-|{\rm d} f_{\varepsilon, j}|\geq 1\quad &\text{on}\ M\setminus\Omega,\label{eq:IDS-band-width-characteristic-1}\\
\frac{m-1}{m} (f_{\varepsilon, j}^2-|{\rm d} f_{\varepsilon, j}|) + \frac{\mathcal{Q}}{2}\geq C_{\varepsilon}\quad &\text{on}\ M, \label{eq:IDS-band-width-characteristic-2}
\end{align}
\end{subequations}
where $C_{\varepsilon}>0$ depends on $\varepsilon,m$ and $\mathcal{Q}$. Now we consider the Callias operator $\widetilde{\mathcal{D}}^{f_{\varepsilon,j},s}$ associated to the Dirac bundle from Example \ref{example:odd-bands} and with the potential $f_{\varepsilon,j}$ subject to the boundary conditions coming from the choice of signs $s(\partial_{\pm}M)=\mp 1$. Since $\ker(\widetilde{\mathcal{D}}^{f_{\varepsilon,j},s}) \neq 0$, we can choose a non-trivial element $u_{\varepsilon, j} \in \ker(\widetilde{\mathcal{D}}^{f_{\varepsilon,j},s})$ for each $j$.
Applying \eqref{eq:null-expansion}, \eqref{eq:zeroterm-odd-bands}, and \eqref{eq:odd-bands-spectral-estimate} in Proposition \ref{pro:odd-band-spectral-estimates}, we find
\begin{equation}\label{eq:band-spectral-estimate}
\begin{aligned}
& \int_{\partial_{-} M} \left(\frac{m-1}{m} f_{\varepsilon, j} - \frac{1}{2} \theta_{\partial_{-} M}^{+} \right) |u_{\varepsilon, j}|^2 dA - \int_{\partial_{+} M} \left( \frac{m-1}{m} f_{\varepsilon, j} + \frac{1}{2} \theta_{\partial_{+} M}^{-} \right) |u_{\varepsilon, j}|^2 dA \\
 \geq & \int_M \left( | \widetilde{\nabla}^{f_{\varepsilon,j}} u_{\varepsilon, j} |^2  + \frac{1}{2} (\mu-|J|_{g_M})|u_{\varepsilon, j}|^2 + \langle \mathscr{R}^{\mathcal{E}} u_{\varepsilon, j}, u_{\varepsilon, j} \rangle \right) dV \\
 & + \frac{m-1}{m} \int_M  (f_{\varepsilon, j}^2-|{\rm d} f_{\varepsilon, j}|+f_{\varepsilon, j} ({\rm tr}_{g_M}k)) |u_{\varepsilon, j}|^2 dV.
\end{aligned}
\end{equation}
Using the same arguments in the proof of Theorem \ref{thm:IDS-long-neck-principle} we have \eqref{eq:boundary-identity}-\eqref{eq:bound-of-first-derivative} for $u_{\varepsilon, j}$ and we may assume $\max_{\Omega}|u_{\varepsilon, j}|=1$. By \eqref{eq:bound-of-spinor} and \eqref{eq:bound-of-first-derivative}, the sequence $u_{\varepsilon, j}$ is uniformly bounded in $H^1(M)$, and thus it has a weak subsequential limit $u_{\varepsilon}$ in $H^1(M)$ with strong convergence in $H^s(M)$ for any $s\in [\frac{1}{2},1)$.
From the trace theorem, the trace map $\tau: H^s(M)\to H^{s-\frac{1}{2}}(\partial M)$ is continuous,
and therefore $u_{\varepsilon, j}$ converges subsequentially to $\tau(u_{\varepsilon})$ in $L^2(\partial M)$. However, since $\Upsilon_{\varepsilon, j}:=\min_{\partial M}\left( \frac{m-1}{m} |f_{\varepsilon, j}|-\frac{1}{2} | \theta_{\partial_{\pm} M}^{\mp}| \right)\to +\infty$, we find that \eqref{eq:bound-of-first-derivative} yields $\tau(u_{\varepsilon})=0$ on $\partial M$, and hence $u_{\varepsilon}\in H_0^1(M^{\circ})$. 
We may assume that ${\rm tr}_{g_M}k\leq 0$ on $M_{-}$ and ${\rm tr}_{g_M}k\geq 0$ on $M_{+}$ by the symmetry of the regions $M_{\pm}$ of a band. Then using \eqref{eq:band-spectral-estimate}, we obtain
\begin{equation}\label{eq:bandspectral}
\begin{aligned}
& \int_{\partial_{-} M} \left(\frac{m-1}{m} f_{\varepsilon, j} - \frac{1}{2} \theta_{\partial_{-} M}^{+} \right) |u_{\varepsilon, j}|^2 dA - \int_{\partial_{+} M} \left( \frac{m-1}{m} f_{\varepsilon, j} + \frac{1}{2} \theta_{\partial_{+} M}^{-} \right) |u_{\varepsilon, j}|^2 dA\\
\geq & \int_{K_{\varepsilon}} \left( \frac{1}{2} (\mu-|J|_{g_M})|u_{\varepsilon, j}|^2 + \langle \mathscr{R}^{\mathcal{E}} u_{\varepsilon, j}, u_{\varepsilon, j} \rangle \right) dV + \int_{M\setminus K_{\varepsilon}}\left( C_{\varepsilon}|u_{\varepsilon, j}|^2 + \langle \mathscr{R}^{\mathcal{E}} u_{\varepsilon, j}, u_{\varepsilon, j} \rangle\right) dV \\
\geq & \int_{K_{\varepsilon} } \left(\frac{1}{2}(\mu-|J|_{g_M})-\|\mathscr{R}^{\mathcal{E}}\|_{\infty}\right)|u_{\varepsilon, j}|^2 dV +\int_{M\setminus K_{\varepsilon}} \left(C_{\varepsilon}- \|\mathscr{R}^{\mathcal{E}}\|_{\infty}\right) |u_{\varepsilon, j}|^2 dV \\
\geq & \int_{\Omega\cap K_{\varepsilon}} \left(\frac{1}{2}(\mu-|J|_{g_M}) -\|\mathscr{R}^{\mathcal{E}}\|_{\infty}\right) |u_{\varepsilon, j}|^2 dV +\int_{\Omega\setminus K_{\varepsilon}} \left(C_{\varepsilon}- \|\mathscr{R}^{\mathcal{E}}\|_{\infty}\right) |u_{\varepsilon, j}|^2 dV.
\end{aligned}
\end{equation}
Then letting $j\to +\infty$ in \eqref{eq:bandspectral} for $u_{\varepsilon, j}$, we see from weak lower semi-continuity of the $H^1$-norm, Fatou's lemma, and strong convergence in $L^2$ that
\begin{equation}\label{eq:bandwidthspectral}
0 \geq  \int_{\Omega\cap K_{\varepsilon}} \left(\frac{1}{2}(\mu-|J|_{g_M}) -\|\mathscr{R}^{\mathcal{E}}\|_{\infty}\right) |u_{\varepsilon}|^2 dV +\int_{\Omega\setminus K_{\varepsilon}} \left( C_{\varepsilon}- \|\mathscr{R}^{\mathcal{E}}\|_{\infty}\right) |u_{\varepsilon}|^2 dV.
\end{equation}
Applying Fatou's lemma to \eqref{eq:bound-of-spinor} and \eqref{eq:bound-of-first-derivative} respectively, we find that $u_{\varepsilon}$ is uniformly bounded in $H^1(M)$. By passing to a subsequence, $u_{\varepsilon} \to u$ in $L^p(M)$ for $p\in [1, \frac{3}{2})$, as $\varepsilon\to 0$.
Note that $K=\cup_{\varepsilon} K_{\varepsilon}$ and $K\subset \Omega$. 
Taking the limit of \eqref{eq:bandwidthspectral} as $\varepsilon\to 0$ implies that
\begin{equation}
0 \geq \int_{K}\left(\frac{1}{2}(\mu-|J|_{g_M})-\|\mathscr{R}^{\mathcal{E}}\|_{\infty}\right)|u|^2 dV + \int_{\Omega\setminus K}\left( C_{\varepsilon} - \|\mathscr{R}^{\mathcal{E}}\|_{\infty}\right)|u|^2 dV.
\end{equation}
Therefore we obtain a contradiction because $\max_{\Omega}|u|=1$. This finishes the proof of Theorem \ref{thm:IDS-band-width-estimate}.
\end{proof}


\begin{proof}[Proof of Theorem \ref{thm:IDS-band-width-estimate}]
By Definition \ref{defn:infinite-vertical-Ahatarea} and Cauchy-Schwarz inequality, for any $\epsilon>0$, we now choose a Hermitian bundle $\mathcal{E}\to M$ satisfying \eqref{eq:infinite-vertical-Ahatarea} and such that the corresponding curvature endomorphism satisfies $\|\mathscr{R}^{\mathcal{E}}\|_{\infty}\leq c_m \|R^{\mathcal{E}}\|_{\infty} < \epsilon$, where $c_m$ is a dimensional constant.
According to \cite[Corollary~3.10]{CZ22} and \eqref{eq:infinite-vertical-Ahatarea}, the index of the Callias operator subject to $s(\partial_{\pm} M)=\mp 1$ is  
\begin{equation}\label{eq:IDS-band-width-non-vanishing-of-index}
{\rm ind}(\widetilde{\mathcal{D}}^{f_{\varepsilon,j},s})={\rm ind}(\slashed{D}_{\partial_{-}M, \left.\mathcal{E}\right|_{\partial_{-}M}}) = \int_{\partial_{-}M} \widehat{A}(\partial_{-}M) \wedge {\rm ch}(\left.\mathcal{E}\right|_{\partial_{-}M}) \neq 0.
\end{equation}
In particular, we have a non-trivial element $u_{\varepsilon, j} \in \ker(\widetilde{\mathcal{D}}^{f_{\varepsilon,j},s})$ for each $j$. Thus our assertion follows.
\end{proof}

\begin{proof}[Proof of Theorem \ref{thm:IDS-KO-band-width-estimate}]
We can define the generalized Callias-type operator $\widetilde{\mathcal{D}}^{f_{\varepsilon,j},s}$ with the choice of signs $s(\partial_{\pm} M)=\mp 1$, see \cite{Cec20,CZ21}. Then we have the similar spectral estimate \eqref{spectralestimate} and Theorem \ref{thm:General-IDS-band-width-estimate} holds for the flat bundle $\mathcal{E}$. Since $(M, g_M)$ is a $\mathcal{KO}$-band, there exists a flat bundle $\mathcal{E}$ over $M$ such that ${\rm ind}(\slashed{D}_{\partial_{-}M,\left.\mathcal{E}\right|_{\partial_{-}M}})\neq 0$. We see from \cite[Corollary~3.10]{CZ22} that
\[
  {\rm ind}(\widetilde{\mathcal{D}}^{f_{\varepsilon,j},s})= {\rm ind}(\slashed{D}_{\partial_{-}M,\left.\mathcal{E}\right|_{\partial_{-}M}})\neq 0.
\]
By Theorem \ref{thm:General-IDS-band-width-estimate}, we finish the proof of Theorem \ref{thm:IDS-KO-band-width-estimate}.
\end{proof}


\section{Width inequalities of geodesic collar neighborhood}\label{sec:width-inequality-of-geodesic-collar-neighborhood}
The purpose of this section is to prove the width inequality in Theorem \ref{thm:IDS-width-inequality}.
The following fact describes the fundamental property of collar neighborhood on an even-dimensional compact spin initial data set that admits infinite relative $\widehat{A}$-area condition. The following basic fact is well-known, see for example \cite[Lemma~6.9]{CZ22}.
\begin{lemma}\label{lem:technical-lemma}
Let $(M,g_M, k)$ be an even-dimensional compact spin initial data set of infinite relative $\widehat{A}$-area and $U\cong \partial M\times (-1,0]\subseteq M$ be an open collar neighborhood. Then for any sufficiently small $\varepsilon_1,\varepsilon_2>0$, there exists a pair of Hermitian vector bundles $(\mathcal{E},\mathcal{F})$ on $M$ such that
\begin{itemize}
\item[(i)] there exists a parallel unitary bundle isomorphism $\Phi:\left.\mathcal{E}\right|_{U_{\varepsilon_1}}\to \left.\mathcal{F}\right|_{U_{\varepsilon_1}}$, where $U_{\varepsilon_1}$ corresponds to $\partial M\times (-1+\varepsilon_1,0]$ in the tubular neighborhood;
\item[(ii)] $\|R^{\mathcal{E}\oplus\mathcal{F}}\|_{\infty}<\varepsilon_2$;
\item[(iii)] ${\rm ind}(\widetilde{\mathcal{D}}^{f_j,-1})\neq 0$, where $\widetilde{\mathcal{D}}^{f_j,-1}$ is the associated Callias operator subject to the boundary condition coming from the sign $s=-1$.
\end{itemize}
\end{lemma}

We next prove the following width inequality of the geodesic collar neighborhood of the boundary.

\begin{proof}[Proof of Theorem \ref{thm:IDS-width-inequality}]
The proof is similar to that of Theorem \ref{thm:IDS-long-neck-principle}. We only point out the necessary modifications. For all sufficiently small $d>0$ denoted by $\mathcal{N}_d$ the open geodesic collar neighborhood of $\partial M$ of width $d$. Suppose, by contradiction, that $\mathcal{N}_{\ell}$ exists and there exists an $\varepsilon>0$ such that $\ell \geq \omega:=\pi\sqrt{ \frac{m-1}{2m\mathcal{Q}} }+\varepsilon$. Fix $\Lambda\in (\omega, \ell)$. Then $K_{\Lambda}:=M\setminus \mathcal{N}_{\Lambda}$ is a compact initial data set with boundary.  
For each $j$, we can define a Lipschitz function $f_j$ as in the equation \eqref{eq:defn-potential}, which satisfy a certain differential inequality and have the property that $f_j\to +\infty$ at $\partial M$ as $j\to +\infty$. Observe that $f_j=0$ on $K_{\Lambda}$. 
Note that the double ${\rm d}M$ has infinite $\widehat{A}$-area if and only if $(M,\partial M)$ has infinite relative $\widehat{A}$-area \cite[Remark~6.6]{CZ22}. Then from Lemma \ref{lem:technical-lemma} and \cite[Corollary~3.9]{CZ22}, there exists a spacetime Gromov-Lawson pair $(\mathcal{E},\mathcal{F})$ and associated Dirac bundle $S\to M$ such that
\begin{itemize}
\item[(i)] $(\mathcal{E},\mathcal{F})$ and thus $S$ have support $K_{\Lambda}$;
\item[(ii)] $2\|\mathscr{R}^{\mathcal{E}\oplus \mathcal{F}}\|_{\infty}<\mu-|J|_{g_M}$ on $K_{\Lambda}^{\circ}$;
\item[(iii)] $\|\mathscr{R}^{\mathcal{E}\oplus \mathcal{F}}\|_{\infty}<C_{\varepsilon}$ on $M\setminus K_{\Lambda}^{\circ}$;
\item[(iv)] ${\rm ind}(\widetilde{\mathcal{D}}^{f_j,-1})={\rm indrel}(M;\mathcal{E},\mathcal{F})\neq 0$ for any potential $f_j$.
\end{itemize}
By (iv), we have a non-trivial element $u_j\in \ker(\widetilde{\mathcal{D}}^{f_j,-1})$. Then it follows from \eqref{eq:energy-momentum}, \eqref{eq:zeroterm-GL}, \eqref{spectralestimate}, Cauchy-Schwarz inequality that 
\begin{equation}\label{eq:gcnspectral}
\begin{aligned}
& -\int_{\partial M} \left(\frac{1}{2}(m-1)H_{\partial M} -\frac{1}{2} |k(\nu,\cdot)^{\top}| + \frac{m-1}{m} f_j \right)|u_j|^2 dA\\
\geq &  \int_M  |\widetilde{\nabla}^{f_j} u_j|^2 dV + \frac{1}{2} \int_M (\mu-|J|_{g_M})|u_j|^2 dV + \int_M \langle \mathscr{R}^{\mathcal{E}\oplus\mathcal{F}} u_j , u_j \rangle  dV\\
& +  \frac{m-1}{m}  \int_M  (f_j^2-|{\rm d} f_j|) |u_j|^2  dV\\
\geq & \int_{K_{\Lambda}^{\circ} } \underbrace{\left(\frac{1}{2}(\mu-|J|_{g_M})-\|\mathscr{R}^{\mathcal{E}\oplus\mathcal{F}}\|_{\infty}\right)}_{>0\ \text{by (ii)} }|u_j|^2 dV +\int_{M\setminus K_{\Lambda}^{\circ}} \underbrace{\left( C_{\varepsilon}-\|\mathscr{R}^{\mathcal{E}\oplus\mathcal{F}}\|_{\infty}\right)}_{>0 \ \text{by (iii)} }|u_j|^2 dV\\
\geq & \int_{\Omega\cap K_{\Lambda}^{\circ}} \left( \frac{1}{2}(\mu-|J|_{g_M})-\|\mathscr{R}^{\mathcal{E}\oplus\mathcal{F}}\|_{\infty}\right) |u_j|^2 dV +\int_{\Omega\setminus K_{\Lambda}^{\circ}} \left( C_{\varepsilon}- \|\mathscr{R}^{\mathcal{E}\oplus\mathcal{F}}\|_{\infty}\right) |u_j|^2 dV.
\end{aligned}
\end{equation}
Furthermore, from the same arguments in the proof of Theorem \ref{thm:IDS-long-neck-principle}, taking the limit of \eqref{eq:gcnspectral} as $j\to +\infty$ with weak lower semi-continuity of the $H^1$-norm, Fatou's lemma, and strong convergence in $L^2$ implies that
\begin{equation}
0\geq \int_{\Omega\cap K_{\Lambda}^{\circ}} \left( \frac{1}{2}(\mu-|J|_{g_M})-\|\mathscr{R}^{\mathcal{E}\oplus\mathcal{F}}\|_{\infty}\right) |u|^2 dV +\int_{\Omega\setminus K_{\Lambda}^{\circ}} \left( C_{\varepsilon}- \|\mathscr{R}^{\mathcal{E}\oplus\mathcal{F}}\|_{\infty}\right) |u|^2 dV.
\end{equation}
Since $\max_{\Omega}|u|=1$, the last two integrands are separately positive, and therefore a contradiction is obtained. We then finish the proof of Theorem \ref{thm:IDS-width-inequality}.
\end{proof}

\section*{Acknowledgements}
The author would like to express his gratitude to Professors Bo Liu and Zhenlei Zhang for their insightful discussions, and helpful comments on an earlier version of this manuscript, and constant encouragements.
The author is also grateful to the anonymous referee for valuable suggestions on the presentation of the manuscript.

\def\cprime{$'$}
\providecommand{\bysame}{\leavevmode\hbox to3em{\hrulefill}\thinspace}
\providecommand{\MR}{\relax\ifhmode\unskip\space\fi MR }
\providecommand{\MRhref}[2]{%
  \href{http://www.ams.org/mathscinet-getitem?mr=#1}{#2}
}
\providecommand{\href}[2]{#2}

\end{document}